\documentclass[11pt,reqno]{amsart}

\usepackage[T1]{fontenc}
\usepackage[utf8]{inputenc}
\usepackage{lmodern}
\usepackage{amsmath,amssymb,amsthm,mathtools,mathrsfs}
\usepackage{geometry}
\usepackage{enumitem}
\usepackage{booktabs}
\usepackage{xcolor}
\usepackage{microtype}
\usepackage[colorlinks=true,linkcolor=blue,citecolor=blue,urlcolor=blue]{hyperref}
\usepackage{aliascnt}
\usepackage[capitalise,noabbrev,nameinlink]{cleveref}
\usepackage{fancyhdr}

\hypersetup{
  pdftitle={Lonely Runners over Function Fields: Quantized Phase--Riesz product},
  pdfauthor={Xiyu Hu},
  pdfsubject={Function-field lonely runners, quantized Fourier support, and phase--Riesz products},
  pdfcreator={pdflatex}
}
\setlist{itemsep=0pt,topsep=0.3em,parsep=0pt,partopsep=0pt}
\allowdisplaybreaks
\numberwithin{equation}{section}
\newtheorem{theorem}{Theorem}[section]
\newaliascnt{conjecture}{theorem}
\newtheorem{conjecture}[conjecture]{Conjecture}
\aliascntresetthe{conjecture}
\newaliascnt{proposition}{theorem}
\newtheorem{proposition}[proposition]{Proposition}
\aliascntresetthe{proposition}
\newaliascnt{lemma}{theorem}
\newtheorem{lemma}[lemma]{Lemma}
\aliascntresetthe{lemma}
\newaliascnt{corollary}{theorem}
\newtheorem{corollary}[corollary]{Corollary}
\aliascntresetthe{corollary}
\newaliascnt{problem}{theorem}
\newtheorem{problem}[problem]{Problem}
\aliascntresetthe{problem}
\theoremstyle{definition}
\newaliascnt{definition}{theorem}
\newtheorem{definition}[definition]{Definition}
\aliascntresetthe{definition}
\theoremstyle{remark}
\newaliascnt{remark}{theorem}
\newtheorem{remark}[remark]{Remark}
\aliascntresetthe{remark}

\crefname{theorem}{Theorem}{Theorems}
\Crefname{theorem}{Theorem}{Theorems}
\crefname{conjecture}{Conjecture}{Conjectures}
\Crefname{conjecture}{Conjecture}{Conjectures}
\crefname{proposition}{Proposition}{Propositions}
\Crefname{proposition}{Proposition}{Propositions}
\crefname{lemma}{Lemma}{Lemmas}
\Crefname{lemma}{Lemma}{Lemmas}
\crefname{corollary}{Corollary}{Corollaries}
\Crefname{corollary}{Corollary}{Corollaries}
\crefname{problem}{Problem}{Problems}
\Crefname{problem}{Problem}{Problems}
\crefname{definition}{Definition}{Definitions}
\Crefname{definition}{Definition}{Definitions}
\crefname{remark}{Remark}{Remarks}
\Crefname{remark}{Remark}{Remarks}

\DeclareMathOperator{\codim}{codim}
\DeclareMathOperator{\supp}{supp}
\DeclareMathOperator{\Tr}{Tr}
\newcommand{\spanq}{\operatorname{span}_{\mathbb F_q}}

\newcommand{\Fq}{\mathbb F_q}
\newcommand{\Fp}{\mathbb F_p}
\newcommand{\Kinf}{\mathbb F_q((T^{-1}))}
\newcommand{\Torus}{\mathbb T_q}
\newcommand{\E}{\mathbb E}
\newcommand{\one}{\mathbf 1}
\newcommand{\cF}{\mathcal F}
\newcommand{\cD}{\mathcal D}
\newcommand{\cR}{\mathcal R}

\newcommand{\cC}{\mathcal C}
\newcommand{\eps}{\varepsilon}
\newcommand{\wh}{\widehat}
\newcommand{\abs}[1]{\lvert#1\rvert}
\newcommand{\norm}[1]{\lVert#1\rVert}
\newcommand{\angles}[1]{\langle#1\rangle}
\newcommand{\set}[1]{\{#1\}}

\title[Quantized phase--Riesz product over function fields]
{Lonely Runners over Function Fields:\\ Quantized Phase--Riesz product}
\author{Xiyu Hu}
\address{School of Mathematical Sciences, University of Chinese Academy of Sciences}
\email{hxypqr@gmail.com}
\date{}

\subjclass[2020]{11J71, 11K41, 05D05, 42B05}
\keywords{Lonely runner conjecture, function fields, Riesz products, hypercontractivity, subspace coverings, sunflowers, polynomial syzygies}

\begin{document}

\begin{abstract}
Let $C_k(q)$ be the least cardinality of a family of nonzero polynomials over $\mathbb F_q$ whose associated codimension-$k$ partial-circulant kernels cover the full coefficient space.  Chow and Rimani\'c conjectured that
\[
 C_k(q)=1+q+\cdots+q^k.
\]
We disprove the unrestricted conjecture by constructing thirteen monic polynomials over $\mathbb F_2$ whose $k=3$ kernels cover $\mathbb F_2^7$; in particular,
\[
 C_3(2)\le13<15.
\]
For a general covering family of size $N=q^k+S$ and $\mathbb F_q$-linear rank $d$, we prove
\[
 S\gg d^{2/3}
 \left(
   \frac{\log(2q)}{\log(eNq^k/S)}
 \right)^{2/3}.
\]
Consequently, for every fixed $k\ge2$ and all sufficiently large $q$,
\[
 C_k(q)\ge q^k+c_kq^{2/3}.
\]
When $k=2$, an integer-multiplicity refinement of the second-moment covering argument yields
\[
 \liminf_{q\to\infty}\frac{C_2(q)-q^2}{q}\ge \widetilde c_2,
\]
where $\widetilde c_2$ is an explicit one-variable variational constant with numerical value
$\widetilde c_2=0.5829944375\ldots$.  We also classify triples admitting two independent low-degree polynomial syzygies and prove a conditional packet-free lower bound of size $q^k+(1/2-o(1))q^{k-1}$.
\end{abstract}

\maketitle
\tableofcontents

\section{Introduction}\label{sec:introduction}

\subsection{The classical Lonely Runner Conjecture}

The Lonely Runner Conjecture originates in work of Wills on simultaneous Diophantine approximation
\cite{Wills1967,Wills1968}.  After learning of Wills's formulation, Cusick gave an equivalent geometric reformulation and introduced the term ``view-obstruction problem''
\cite{Cusick}.  Its standard name comes from the formulation attributed to Goddyn in the paper of Bienia, Goddyn, Gvozdjak, Seb\H{o}, and Tarsi
\cite{BieniaEtAl}.  Suppose that runners with distinct constant speeds start together on a unit circular track.  The conjecture asserts that each runner is, at some time, at circular distance at least the reciprocal of the number of runners from every other runner.

After subtracting the speed of the runner under consideration, and then reducing from real to rational and integer speeds, one obtains the stationary-runner formulation.  For a finite set $V=\{v_1,\ldots,v_n\}\subset\mathbb N$ of distinct positive integers, define
\begin{equation}\label{eq:ML-classical}
 \operatorname{ML}(V)
 :=\sup_{t\in\mathbb R/\mathbb Z}
   \min_{v\in V}\norm{tv}_{\mathbb R/\mathbb Z}.
\end{equation}
The conjecture is
\begin{equation}\label{eq:LRC-classical}
 \operatorname{ML}(V)\ge\frac1{n+1}.
\end{equation}
The constant is best possible, since $V=\{1,2,\ldots,n\}$ has loneliness $1/(n+1)$.  Here $n$ is the number of non-stationary speeds, so the original formulation has $n+1$ runners.

The problem has useful interpretations in graph colouring, nowhere-zero flows, view obstruction, billiard motion, and covering radii of lattice zonotopes.  The distance-graph viewpoint and its circular and fractional colouring parameters are surveyed by Liu \cite{LiuSurvey}; geometric and polyhedral formulations were developed further by Henze--Malikiosis and Beck--Ho\c{s}ten--Schymura \cite{HenzeMalikiosis,BeckHostenSchymura}.  Rifford related quantitative bounds for the first time at which a runner becomes lonely to a finite-dimensional covering problem \cite{Rifford}.  A complementary inverse direction studies the set of attainable maximum-loneliness values: Kravitz proposed a rigidity refinement, while Giri and Kravitz determined the accumulation structure of the resulting spectra \cite{Kravitz,GiriKravitz}.  See also the recent survey \cite{PerarnauSerraSurvey}.

For fixed dimension, Betke and Wills settled the case of three positive speeds, Cusick and Pomerance the case of four, Bohman, Holzman, and Kleitman the case of five, and Barajas and Serra the case of six \cite{BetkeWills,CusickPomerance,BohmanHolzmanKleitman,BarajasSerra}; Renault later gave a shorter view-obstruction proof of the five-speed case \cite{Renault}.  Under additional hypotheses on the speeds, Dubickas proved the conjecture for an explicit class of sufficiently lacunary sequences \cite{Dubickas}, while random speed sets typically have much larger loneliness than the conjectured worst-case scale \cite{CzerwinskiRandom}.  A linearly-exponential finite-checking theorem of Malikiosis, Santos, and Schymura \cite{MalikiosisSantosSchymura} made substantially larger computer-assisted verifications possible.  Rosenfeld established the eight-runner case and subsequently the nine-runner case \cite{Rosenfeld,RosenfeldNine}; Trakulthongchai proved the nine- and ten-runner cases \cite{Trakulthongchai}; and Sungkawichai and Trakulthongchai extended the verified range through eleven, twelve, and thirteen runners \cite{SungkawichaiTrakulthongchai}.

A separate direction asks for the best lower bound that holds uniformly as $n\to\infty$.  For $v\ne0$ and $0<\delta<1/2$, let
\[
 B(v;\delta):=\set{t\in\mathbb R/\mathbb Z:\norm{tv}_{\mathbb R/\mathbb Z}\le\delta}.
\]
If $\operatorname{ML}(V)\le\delta$, then the sets $B(v;\delta)$ cover the circle and each has measure $2\delta$.  The union bound therefore gives
\[
 \operatorname{ML}(V)\ge\frac1{2n}.
\]
Chen and Chen--Cusick obtained the first refinements at the $n^{-2}$ scale, and Perarnau and Serra proved the unrestricted asymptotic estimate
\[
 \operatorname{ML}(V)\ge\frac1{2n-2+o(1)}
 =\frac1{2n}+\frac1{2n^2}+o(n^{-2});
\]
see \cite{Chen,ChenCusick,PerarnauSerra}.  Tao subsequently obtained
\[
 \operatorname{ML}(V)
 \ge\frac1{2n}
 +\frac{c\log n}{n^2(\log\log n)^2}
\]
for all sufficiently large $n$ \cite{Tao}.  Bedert introduced an additive-dimension dichotomy and Riesz products to prove the first polynomial improvement,
\[
 \operatorname{ML}(V)
 \ge\frac1{2n}+\frac1{n^{5/3+o(1)}}
\]
\cite{Bedert}.  A phase-randomised refinement of the large-dimension branch improves the exponent to $13/8$; this result is recorded in the joint online note
\cite{BedertHu}.

\subsection{The function-field model}

Finite-field variants of the runner problem were considered by Czerwi\'nski and Grytczuk
\cite{CzerwinskiGrytczuk}.  The model studied here is the non-Archimedean analogue introduced by Chow and Rimani\'c
\cite{ChowRimanic}.  Fix a prime power $q$ and let
\[
 \Kinf=\Fq((T^{-1})),
 \qquad
 \Torus=T^{-1}\Fq[[T^{-1}]].
\]
Every $\beta\in\Kinf$ decomposes uniquely as
\[
 \beta=[\beta]+\{\beta\},
 \qquad
 [\beta]\in\Fq[T],
 \quad
 \{\beta\}\in\Torus.
\]
If $\beta=\sum_{j\le m}b_jT^j$ with $b_m\ne0$, set
$\abs\beta=q^m$ and $\abs0=0$, and define the distance to the polynomial ring by
\[
 \norm\beta:=\abs{\{\beta\}}.
\]
Thus $\norm\beta$ takes values in
$\set{0,q^{-1},q^{-2},\ldots}$ and satisfies the ultrametric inequality.

For a finite family
$\cF\subset\Fq[T]\setminus\set{0}$, define
\begin{equation}\label{eq:delta-def-intro}
 \delta(\cF)
 :=\sup_{\alpha\in\Torus}
   \min_{f\in\cF}\norm{\alpha f}.
\end{equation}
Multiplication of a speed by an element of $\Fq^\times$ does not change $\delta(\cF)$, so speeds may be normalised to be monic.  For $k\ge1$, put
\begin{equation}\label{eq:Qk-def-intro}
 Q_k(q):=1+q+\cdots+q^k
 =\frac{q^{k+1}-1}{q-1}.
\end{equation}
The family of all monic polynomials of degree at most $k$ has cardinality $Q_k(q)$ and loneliness at most $q^{-(k+1)}$.  This led Chow and Rimani\'c to the following conjecture.

\begin{conjecture}[Chow--Rimani\'c]\label{conj:CR}
If $\abs{\cF}<Q_k(q)$, then
\[
 \delta(\cF)\ge q^{-k}.
\]
\end{conjecture}

It is convenient to encode the problem by
\begin{equation}\label{eq:Ck-def}
 C_k(q):=
 \min\set{
   \abs{\cF}:
   \cF\subset\Fq[T]\setminus\set{0},\
   \delta(\cF)<q^{-k}
 }.
\end{equation}
The discreteness of the norm means that
$\delta(\cF)<q^{-k}$ is equivalent to
$\delta(\cF)\le q^{-(k+1)}$.  The union bound gives
$C_k(q)\ge q^k+1$, while the standard family gives
$C_k(q)\le Q_k(q)$.  Hence \cref{conj:CR} is precisely the assertion
$C_k(q)=Q_k(q)$.

Chow and Rimani\'c proved the union-bound theorem
$\abs{\cF}\le q^k\Rightarrow\delta(\cF)\ge q^{-k}$, established the conjecture under a small-degree hypothesis, and obtained in the first nontrivial case $k=2$ the estimate
\begin{equation}\label{eq:CR-k2-intro}
 C_2(q)\ge q^2+0.4877q-O(1).
\end{equation}
Their $k=2$ proof is based on the geometry of codimension-four sunflowers among partial-circulant kernels.

\subsection{Main results}

Our first result shows that the unrestricted all-$q$, all-$k$ conjecture is false.

\begin{theorem}\label{thm:counterexample-intro}
There is a family $\cF_*\subset\mathbb F_2[T]$ of thirteen monic polynomials satisfying
\[
 \delta(\cF_*)=2^{-4}.
\]
Consequently,
\[
 C_3(2)\le13<15=Q_3(2).
\]
\end{theorem}

The example was initially found computationally using a deep cross-entropy search in the spirit of Adam Zsolt Wagner's method \cite{WagnerNeural}, and is directly verifiable afterwards.

The family is displayed in \cref{sec:counterexample}.  Its covering property is certified by checking the $2^7$ vectors of the relevant coefficient space; the accompanying file
\texttt{verify\_counterexample.py} performs this check.

For a family $\cF$, write
\[
 d(\cF):=\dim_{\Fq}\spanq\cF.
\]
The main general estimate is sensitive to this linear rank.

\begin{theorem}\label{thm:rank-intro}
Fix $k\ge2$.  Suppose that $\delta(\cF)<q^{-k}$ and write
\[
 N:=\abs{\cF},
 \qquad
 S:=N-q^k,
 \qquad
 d:=d(\cF),
 \qquad
 L:=\log\left(\frac{eNq^k}{S}\right).
\]
Then $S>0$ and
\begin{equation}\label{eq:rank-intro}
 S\gg
 d^{2/3}
 \left(\frac{\log(2q)}{L}\right)^{2/3},
\end{equation}
where the implied constant is absolute.
\end{theorem}

A compression argument removes the rank parameter when $k$ is fixed.

\begin{theorem}\label{thm:fixed-k-intro}
For every fixed integer $k\ge2$, there are constants $c_k>0$ and $q_0(k)$ such that
\begin{equation}\label{eq:fixed-k-intro}
 C_k(q)\ge q^k+c_kq^{2/3}
\end{equation}
for every prime power $q\ge q_0(k)$.
\end{theorem}

For $k=2$, a different argument gives a stronger estimate on the natural linear scale.  For $u\ge0$, define
\begin{equation}\label{eq:vartheta-def}
 \vartheta(u):=
 \max_{h\in\mathbb N}
 \frac{2h-u}{h(h+1)}.
\end{equation}
Equivalently, if $h-1\le u\le h$, then
\begin{equation}\label{eq:vartheta-piecewise}
 \vartheta(u)=\frac{2h-u}{h(h+1)}.
\end{equation}
For $0<\lambda<1$, put
\begin{align}
 \widetilde a(x)
 &:=x\vartheta(x),\label{eq:atilde-def}\\
 \widetilde b_\lambda(x)
 &:=\lambda+(x-\lambda)(1-\lambda)
 \vartheta\left(\frac{x-\lambda}{1-\lambda}\right),
 \label{eq:btilde-def}\\
 \widetilde{\cC}(\lambda)
 &:=\lambda^3+
 \lambda\int_\lambda^{1/\lambda}
 \max\set{\widetilde a(x),\widetilde b_\lambda(x)}\,dx,
 \label{eq:Ctilde-lambda}
\end{align}
and extend continuously by
$\widetilde{\cC}(0)=\widetilde{\cC}(1)=1$.  Define
\begin{equation}\label{eq:c2tilde}
 \widetilde c_2
 :=\min_{0\le\lambda\le1}\widetilde{\cC}(\lambda).
\end{equation}

\begin{theorem}\label{thm:c2-intro}
As $q\to\infty$ through prime powers,
\begin{equation}\label{eq:c2-intro}
 \liminf_{q\to\infty}
 \frac{C_2(q)-q^2}{q}
 \ge\widetilde c_2.
\end{equation}
Numerically,
\[
 \widetilde c_2=0.5829944375\ldots,
 \qquad
 \widetilde\lambda_*=0.4067490\ldots.
\]
Moreover, $\widetilde c_2>0.5278$, so the theorem gives a rigorous improvement over \eqref{eq:CR-k2-intro}.
\end{theorem}

The exact result is the variational formula \eqref{eq:c2tilde}; the decimal value is included only for orientation and is reproduced by the accompanying file
\texttt{evaluate\_c2\_constant.py}.

The final unconditional result identifies the first algebraic obstruction to a bound on the $q^{k-1}$ scale.  Let
$A_k:=\Fq[T]_{<k}$ and define
\begin{equation}\label{eq:Rk-intro}
 \cR_k(m,f,g)
 :=\set{(A,B,C)\in A_k^3:Am+Bf+Cg=0}.
\end{equation}

\begin{theorem}\label{thm:syzygy-intro}
If $\dim_{\Fq}\cR_k(m,f,g)\ge2$, then at least one of the following holds:
\begin{enumerate}[label=\textup{(\roman*)}]
 \item there is a nonzero relation
 $Am+Bf+Cg=0$ with
 $\deg A,\deg B,\deg C\le k-2$;
 \item there are a nonzero polynomial $P$ and polynomials
 $m_0,f_0,g_0$ of degree at most $2k-2$ such that
 \[
  m=Pm_0,
  \qquad
  f=Pf_0,
  \qquad
  g=Pg_0.
 \]
\end{enumerate}
\end{theorem}

Under an explicit quantitative hypothesis excluding a positive density of incoherent packets of these two types, we prove that every such covering family has size
$q^k+(1/2-o(1))q^{k-1}$.  The conditional statement appears in \cref{sec:packets}.

\subsection{The method}

The first step is the partial-circulant covering formulation of Chow and Rimani\'c.  If all speeds have degree at most $D$, set
\[
 V_{D+k}:=\Fq[T]_{<D+k},
 \qquad
 U_f:=fA_k,
 \qquad
 K_f:=U_f^\perp.
\]
Then $K_f$ has codimension $k$, and
\[
 \delta(\cF)<q^{-k}
 \quad\Longleftrightarrow\quad
 \bigcup_{f\in\cF}K_f=V_{D+k}.
\]
Thus a counterexample to the desired loneliness estimate is exactly a cover by a restricted family of codimension-$k$ subspaces.

Let
\[
 \Phi(x):=\sum_{f\in\cF}\one_{K_f}(x)
\]
be the covering multiplicity.  Its mean is
$Nq^{-k}=1+S/q^k$.  The central advantage of the function-field setting is the exact Fourier identity
\begin{equation}\label{eq:exact-Fourier-intro}
 \wh\Phi(\xi)
 =q^{-k}\abs{\set{f\in\cF:\xi\in U_f}}.
\end{equation}
Consequently, every Fourier coefficient is nonnegative, every nonzero coefficient is at least $q^{-k}$, and the Fourier support is the finite union
$\bigcup_fU_f$.  Positivity, quantisation, and finite support are the three features that allow the argument to treat all Fourier levels simultaneously.

Choose linearly independent speeds $m_1,\ldots,m_d$ and consider the surjective map
\[
 \Lambda(x)
 :=(\angles{m_1,x},\ldots,\angles{m_d,x})
 \in\Fq^d.
\]
For $b\in\Fq^d$, the density
\[
 R_b^+(x):=q^d\one_{\{\Lambda(x)=b\}}
\]
is a positive $q$-ary phase product of mean one.  The covering inequality gives a nonnegative function
\[
 Y(b):=\E_x\Phi(x)R_b^+(x)-1
\]
on $\Fq^d$ with mean $S/q^k$, and its nonzero Fourier coefficients are exactly
\[
 \wh Y(a)=\wh\Phi\left(\sum_i a_im_i\right).
\]
Wolff's sharp dependence of finite-space hypercontractivity on the least atom
\cite{Wolff} gives the level estimate
\begin{equation}\label{eq:level-energy-intro}
 \sum_{|\supp a|=s}
 \abs{\wh\Phi\left(\sum_i a_im_i\right)}^2
 \ll
 \left(\frac{S}{q^k}\right)^2
 \left(
   \frac{CqL}{\log(2q)}
 \right)^s.
\end{equation}

The negative product is built from the one-coordinate density
\[
 r_\rho(y)
 :=1-\frac{\rho}{q-1}
 \sum_{c\in\Fq^\times}\psi(cy),
\]
which equals $1-\rho$ at $y=0$ and
$1+\rho/(q-1)$ otherwise.  Testing $\Phi\ge1$ against
$\prod_i r_\rho(\angles{m_i,x})$, and using Fourier quantisation to convert \eqref{eq:level-energy-intro} from $\ell^2$ to $\ell^1$, gives
\begin{equation}\label{eq:master-intro}
 0\le
 S-c\rho d
 +CS^2\rho^2
 \left(\frac{L}{\log(2q)}\right)^2.
\end{equation}
Optimising in $\rho$ proves \cref{thm:rank-intro}.  The factor $q$ in the $q$-ary hypercontractive estimate is cancelled by the coefficient $(q-1)^{-1}$ in the negative product.  This cancellation is what removes the logarithmic loss from the fixed-$k$ theorem.

To prove \cref{thm:fixed-k-intro}, we must rule out low-rank covers.  An arbitrary vector-space compression would destroy the blocks $fA_k$.  Instead, reduction modulo an irreducible polynomial gives a map compatible with multiplication by $T$.  Iterating the construction transforms a rank-$d$ cover into one whose maximum degree is less than $(k-1)d$.  The small-degree theorem of Chow and Rimani\'c then forces $d\gg_kq$, after which \cref{thm:rank-intro} gives an excess of order $q^{2/3}$.

It is useful to compare this route with Tao's argument in the classical setting.  There, large third moments of the Bohr-covering multiplicity are converted, through Bohr-set--progression duality and pigeonholing, into relations
\[
 a_\ell v_\ell=n_{i,\ell}v_i+n_{j,\ell}v_j
\]
with $a_\ell=o(\log n)$.  The denominators therefore range over a polylogarithmic set, and after placing the numerators in one short progression a sieve for medium-sized prime factors forces additional overlaps \cite{Tao}.  The moment-to-relation step is formally cleaner over function fields, since rank defects give polynomial syzygies $Af+Bg+Ch=0$ with $A,B,C\in A_k$.  The next step is harder, however: low degree does not mean few labels when $q$ grows, since already one coefficient ranges over $q^k$ possibilities; scalar multiples describe the same projective relation; and labels attached to overlapping triples need not agree.  Thus the local syzygies do not automatically assemble into a common short one-dimensional model on which Tao's pigeonhole and sieve can operate.  Bedert's Riesz-product method bypasses this labelled-agreement obstruction in the large-rank branch: it tests the pointwise cover against a probability density whose first Fourier layer accumulates over many independent directions, while dissociation---and here exact Fourier quantisation---controls the higher layers \cite{Bedert}.  The remaining syzygy-agreement problem is discussed in \cref{sec:packets}.

The proof of \cref{thm:c2-intro} is independent of the phase argument.  Its new input is the pointwise inequality
\begin{equation}\label{eq:integer-multiplicity-intro}
 \one_{\{z>0\}}
 \ge
 \frac{2z}{h+1}
 -\frac{z(z-1)}{h(h+1)}
 \qquad
 (z\in\mathbb Z_{\ge0},\ h\in\mathbb N).
\end{equation}
Unlike Cauchy--Schwarz, this inequality uses that a covering multiplicity is integer-valued.  Optimising over $h$ produces the function $\vartheta$ in \eqref{eq:vartheta-def}, and hence the improved envelopes
$\widetilde a$ and $\widetilde b_\lambda$ throughout the Chow--Rimani\'c sunflower ordering.

Finally, higher moments of the covering multiplicity detect tuples for which the sum of the blocks $fA_k$ has unexpectedly small rank.  Such rank defects carry coefficient labels in $A_k$.  The two-syzygy theorem identifies the first possible packet types; the remaining difficulty is a labelled local-to-global agreement theorem ensuring that many local relations belong to a controlled collection of coherent packets.

\subsection{Status of the statements and organisation}

Theorems \ref{thm:counterexample-intro}--\ref{thm:syzygy-intro} are unconditional.  The counterexample is computer-assisted only through an exhaustive check of $128$ vectors by a short standard-library Python program.  The numerical approximation to $\widetilde c_2$ is not used in the exact theorem.  The half-way result in \cref{sec:packets} is explicitly conditional on the packet-free hypothesis stated there.

\Cref{sec:covering} develops the function-field and Fourier formalism.  The counterexample is proved in \cref{sec:counterexample}.  The $q$-ary quantized phase--Riesz estimate is established in \cref{sec:phase}, and the compression argument is given in \cref{sec:compression}.  The sharper $k=2$ estimate occupies \cref{sec:k2}.  Syzygies and the conditional packet theorem are treated in \cref{sec:packets}.  We conclude with further questions and a statement on the accompanying code.

\section{The covering model and exact Fourier support}\label{sec:covering}

\subsection{Coefficient spaces and partial-circulant kernels}

For an integer $M\ge1$, let
\[
 V_M:=\Fq[T]_{<M}
 =\set{g\in\Fq[T]:\deg g<M}.
\]
We identify $V_M$ with $\Fq^M$ by the coefficient map and equip it with the nondegenerate bilinear form
\begin{equation}\label{eq:coefficient-pairing}
 \angles{g,h}
 :=\sum_{j=0}^{M-1}g_jh_j,
 \qquad
 g=\sum_{j=0}^{M-1}g_jT^j,
 \quad
 h=\sum_{j=0}^{M-1}h_jT^j.
\end{equation}
Fix $D\ge0$ and $k\ge1$, and suppose that every speed under consideration has degree at most $D$.  Set
\begin{equation}\label{eq:Uf-Kf}
 A_k:=\Fq[T]_{<k},
 \qquad
 U_f:=fA_k
 =\spanq\set{f,Tf,\ldots,T^{k-1}f}
 \le V_{D+k},
 \qquad
 K_f:=U_f^\perp.
\end{equation}
Multiplication by a nonzero polynomial is injective, so
\begin{equation}\label{eq:dimensions-U-K}
 \dim_{\Fq}U_f=k,
 \qquad
 \codim_{V_{D+k}}K_f=k,
 \qquad
 \abs{K_f}=q^D.
\end{equation}
The equations defining $K_f$ are the first $k$ rows of a circulant convolution matrix, which explains the term \emph{partial-circulant kernel}.

\begin{lemma}\label{lem:covering}
Let $\cF$ be a family of nonzero polynomials of degree at most $D$.  Then
\[
 \delta(\cF)\ge q^{-k}
 \quad\Longleftrightarrow\quad
 \bigcup_{f\in\cF}K_f\ne V_{D+k}.
\]
Equivalently,
\begin{equation}\label{eq:covering-equivalence}
 \delta(\cF)<q^{-k}
 \quad\Longleftrightarrow\quad
 \bigcup_{f\in\cF}K_f=V_{D+k}.
\end{equation}
\end{lemma}

\begin{proof}
Write
\[
 \alpha=\sum_{j=1}^{D+k}x_jT^{-j}
 +O(T^{-(D+k+1)})
\]
and let
$f=\sum_{i=0}^Da_iT^i$.  The coefficient of $T^{-s}$ in the fractional part of $\alpha f$ is
\begin{equation}\label{eq:partial-circulant-equations}
 \sum_{i=0}^Da_ix_{i+s},
 \qquad 1\le s\le k.
\end{equation}
If $x:=\sum_{j=0}^{D+k-1}x_{j+1}T^j\in V_{D+k}$, then the expression in \eqref{eq:partial-circulant-equations} is
$\angles{x,T^{s-1}f}$.  Hence
\[
 \norm{\alpha f}<q^{-k}
 \quad\Longleftrightarrow\quad
 x\in K_f.
\]
The coefficients of $\alpha$ beyond $T^{-(D+k)}$ do not affect these $k$ equations.  Therefore a time $\alpha$ with
$\norm{\alpha f}\ge q^{-k}$ for every $f\in\cF$ exists exactly when the corresponding coefficient vector lies outside every $K_f$.
\end{proof}

The covering formulation immediately recovers the elementary bounds in the introduction.

\begin{corollary}\label{cor:basic-Ck-bounds}
For every $q$ and $k\ge1$,
\begin{equation}\label{eq:basic-Ck-bounds}
 q^k+1\le C_k(q)\le Q_k(q).
\end{equation}
\end{corollary}

\begin{proof}
Every $K_f$ has density $q^{-k}$ in $V_{D+k}$.  Thus a cover by $N$ kernels requires $Nq^{-k}\ge1$, and hence $N\ge q^k$.  Equality cannot occur: if $N=q^k$, then the covering multiplicity has mean one and is pointwise at least one, so it is identically one; at the origin, however, all $N$ kernels meet and the multiplicity equals $N>1$.  Therefore $N\ge q^k+1$.

For the upper bound, fix $\alpha\in\Torus$.  The conditions that the coefficients of
$T^{-1},\ldots,T^{-k}$ in $\alpha f$ vanish are $k$ homogeneous linear equations in the $k+1$ coefficients of
$f\in\Fq[T]_{\le k}$.  They admit a nonzero solution, which may be normalised to a monic polynomial of degree at most $k$.  Hence the family of all monic polynomials of degree at most $k$ has loneliness at most $q^{-(k+1)}$ and cardinality $Q_k(q)$.
\end{proof}

\subsection{The covering multiplicity}

Assume from now on that the kernels associated with $\cF$ cover $V_{D+k}$.  Write
\begin{equation}\label{eq:Phi-def}
 \Phi(x):=\sum_{f\in\cF}\one_{K_f}(x),
 \qquad
 N:=\abs{\cF},
 \qquad
 S:=N-q^k,
 \qquad
 \eps:=\frac{S}{q^k}.
\end{equation}
Then
\begin{equation}\label{eq:Phi-cover-mean}
 1\le\Phi(x)\le N,
 \qquad
 \E_x\Phi(x)=Nq^{-k}=1+\eps.
\end{equation}
In particular, $S>0$.  Indeed, if $S=0$, then
$\Phi\ge1$ and $\E\Phi=1$, so $\Phi\equiv1$, whereas
$\Phi(0)=N=q^k>1$.

The pointwise inequality in \eqref{eq:Phi-cover-mean} may be tested against any probability density.  Namely, if
$R:V_{D+k}\to[0,\infty)$ and $\E R=1$, then
\begin{equation}\label{eq:probability-density-test}
 1\le\E_x\Phi(x)R(x).
\end{equation}
The uniform choice $R\equiv1$ gives the union bound.  The role of a Riesz product is to construct a nonnegative density with a negative first Fourier layer while controlling the higher layers created by positivity.

\subsection{Fourier transform}

Write $q=p^r$ and fix the canonical nontrivial additive character
\[
 \psi(u):=
 \exp\left(\frac{2\pi i}{p}\Tr_{\Fq/\Fp}(u)\right).
\]
For $\xi,x\in V_{D+k}$, put
\[
 \chi_\xi(x):=\psi(\angles{\xi,x}),
 \qquad
 \wh g(\xi):=\E_xg(x)\overline{\chi_\xi(x)}.
\]

\begin{lemma}\label{lem:fourier-subspace}
For every nonzero polynomial $f$ of degree at most $D$ and every
$\xi\in V_{D+k}$,
\begin{equation}\label{eq:fourier-subspace}
 \wh{\one_{K_f}}(\xi)=q^{-k}\one_{U_f}(\xi).
\end{equation}
Consequently,
\begin{equation}\label{eq:Phi-hat}
 \wh\Phi(\xi)=q^{-k}r_{\cF}(\xi),
 \qquad
 r_{\cF}(\xi):=
 \abs{\set{f\in\cF:\xi\in U_f}}.
\end{equation}
In particular,
\begin{equation}\label{eq:quantisation}
 \wh\Phi(\xi)\ge0,
 \qquad
 \wh\Phi(\xi)>0\Longrightarrow\wh\Phi(\xi)\ge q^{-k},
 \qquad
 \supp\wh\Phi=\bigcup_{f\in\cF}U_f.
\end{equation}
\end{lemma}

\begin{proof}
The annihilator of $K_f=U_f^\perp$ is $U_f$.  Averaging a character over $K_f$ gives the density $q^{-k}$ when the character is trivial on $K_f$, and zero otherwise.  This proves \eqref{eq:fourier-subspace}; summing over $f$ gives \eqref{eq:Phi-hat} and \eqref{eq:quantisation}.
\end{proof}

The exact formula \eqref{eq:Phi-hat} is the decisive simplification relative to the real-circle problem.  The Fourier transform is finitely supported, has a fixed positive quantum, and carries no signs.  These properties will be used simultaneously in \cref{sec:phase}.

\section{A counterexample over \texorpdfstring{$\mathbb F_2$}{F2}}\label{sec:counterexample}

We now prove \cref{thm:counterexample-intro}.  Let
\begin{align*}
\cF_*:=\{\,&1,\ T^2+T,\ T^2+T+1,\ T^3,\\
& T^3+T^2+1,\ T^3+T+1,\ T^3+T^2+T+1,\\
& T^4+T^2,\ T^4+T,\ T^4+T^3+1,\\
& T^4+T^2+1,\ T^4+T+1,\ T^4+T^3+T^2+T+1\,\}.
\end{align*}
Thus $\abs{\cF_*}=13$, the maximum degree is $D=4$, and we take $k=3$.  The coefficient space is
$V_7\cong\mathbb F_2^7$.

If
\[
 f=a_0+a_1T+\cdots+a_4T^4,
 \qquad
 x=(x_0,\ldots,x_6)\in\mathbb F_2^7,
\]
then $x\in K_f$ exactly when
\begin{equation}\label{eq:counter-kernel}
 \sum_{j=0}^4a_jx_j=0,
 \qquad
 \sum_{j=0}^4a_jx_{j+1}=0,
 \qquad
 \sum_{j=0}^4a_jx_{j+2}=0.
\end{equation}

\begin{proposition}\label{prop:certificate}
The thirteen kernels associated with $\cF_*$ cover $\mathbb F_2^7$.  Their covering-multiplicity distribution is
\begin{center}
\begin{tabular}{c|rrrrrr}
\toprule
Multiplicity & $1$ & $2$ & $3$ & $4$ & $7$ & $13$\\
\midrule
Number of vectors & $83$ & $26$ & $15$ & $2$ & $1$ & $1$\\
\bottomrule
\end{tabular}
\end{center}
In particular, no vector has multiplicity zero.
\end{proposition}

\begin{proof}
Substitution of the thirteen coefficient vectors into
\eqref{eq:counter-kernel}, followed by enumeration of the
$2^7=128$ possible vectors, gives the displayed distribution.  The accompanying file \texttt{verify\_counterexample.py} performs exactly this check using only bit operations and the Python standard library.  As a consistency check, the total incidence count is
\[
 83+2\cdot26+3\cdot15+4\cdot2+7+13
 =208
 =13\cdot2^4,
\]
which is the sum of the sizes of the thirteen four-dimensional kernels.
\end{proof}

\begin{proof}[Proof of \cref{thm:counterexample-intro}]
By \cref{lem:covering,prop:certificate},
$\delta(\cF_*)<2^{-3}$, and the discreteness of the norm gives
$\delta(\cF_*)\le2^{-4}$.

For the reverse inequality, take $\alpha=T^{-4}$.  Every member of
$\cF_*$ has degree at most four, and $T^4$ itself is not in the family.  Hence the fractional part of $T^{-4}f$ is nonzero and involves only the powers
$T^{-1},\ldots,T^{-4}$; consequently
$\norm{T^{-4}f}\ge2^{-4}$ for every $f\in\cF_*$.  The speed $f=1$ gives equality.  Therefore
\[
 \min_{f\in\cF_*}\norm{T^{-4}f}=2^{-4},
\]
which proves $\delta(\cF_*)=2^{-4}$.  Finally,
$13<1+2+4+8=15$.
\end{proof}

\begin{remark}
The counterexample occurs in characteristic two and shows that the original all-$q$, all-$k$ formulation is false.  It does not conflict with any of the sufficiently-large-$q$ results proved below.  A corrected conjecture may exclude finitely many pairs $(q,k)$ or may assert the projective threshold only after $q$ is sufficiently large in terms of $k$.
\end{remark}

\section{Quantized \texorpdfstring{$q$}{q}-ary phase--Riesz products}\label{sec:phase}

We prove \cref{thm:rank-intro}.  The positive product below is supported on fibres of a linear map to $\Fq^d$.  It converts the covering excess into Fourier energy on every Hamming level.  A second product has a negative first level and coefficients of size $(q-1)^{-s}$ at level $s$.

Let
\begin{equation}\label{eq:D-independent}
 \cD=\set{m_1,\ldots,m_d}\subset\cF
\end{equation}
be linearly independent over $\Fq$, with $d=d(\cF)$.  For
$a=(a_1,\ldots,a_d)\in\Fq^d$, write
\begin{equation}\label{eq:Aa-def}
 Aa:=\sum_{i=1}^da_im_i,
 \qquad
 |a|_0:=\abs{\set{i:a_i\ne0}}.
\end{equation}
The map $a\mapsto Aa$ is injective.

\subsection{The positive phase product}

Define
\begin{equation}\label{eq:Lambda-def}
 \Lambda:V_{D+k}\longrightarrow\Fq^d,
 \qquad
 \Lambda(x):=
 (\angles{m_1,x},\ldots,\angles{m_d,x}).
\end{equation}
The nondegeneracy of the coefficient pairing and the linear independence of the $m_i$ imply that $\Lambda$ is surjective.  For $b=(b_1,\ldots,b_d)\in\Fq^d$, set
\begin{equation}\label{eq:qary-positive-product}
 R_b^+(x)
 :=q^d\one_{\{\Lambda(x)=b\}}.
\end{equation}
Equivalently,
\begin{equation}\label{eq:qary-positive-product-expanded}
 R_b^+(x)
 =\prod_{i=1}^d
 \left(
   \sum_{c\in\Fq}
   \psi\bigl(c(\angles{m_i,x}-b_i)\bigr)
 \right).
\end{equation}

\begin{lemma}\label{lem:qary-positive}
For every $b\in\Fq^d$,
\[
 R_b^+(x)\ge0,
 \qquad
 \E_xR_b^+(x)=1.
\]
\end{lemma}

\begin{proof}
The character sum in the $i$th factor of
\eqref{eq:qary-positive-product-expanded} equals $q$ when
$\angles{m_i,x}=b_i$ and zero otherwise, proving
\eqref{eq:qary-positive-product}.  Since $\Lambda$ is surjective, every fibre has density $q^{-d}$, and the asserted normalisation follows.
\end{proof}

Define
\begin{equation}\label{eq:Y-def}
 Y(b):=\E_x\Phi(x)R_b^+(x)-1.
\end{equation}
For $G:\Fq^d\to\mathbb C$, use the Fourier convention
\begin{equation}\label{eq:phase-Fourier}
 \wh G(a)
 :=\E_{b\in\Fq^d}
 G(b)\overline{\psi(a\cdot b)},
 \qquad
 a\cdot b:=\sum_{i=1}^da_ib_i.
\end{equation}

\begin{lemma}\label{lem:Y-small}
The function $Y$ satisfies
\begin{equation}\label{eq:Y-small}
 0\le Y(b)\le N-1,
 \qquad
 \E_bY(b)=\eps.
\end{equation}
Moreover, for every nonzero $a\in\Fq^d$,
\begin{equation}\label{eq:qary-phase-identification}
 \wh Y(a)=\wh\Phi(Aa).
\end{equation}
\end{lemma}

\begin{proof}
By \cref{lem:qary-positive}, $R_b^+$ is a probability density in $x$.  Since $1\le\Phi\le N$, its average against this density lies in $[1,N]$, proving the pointwise bounds.  Averaging the fibre densities over $b$ gives
$\E_bR_b^+(x)=1$, and hence
\[
 \E_bY(b)=\E_x\Phi(x)-1=\eps.
\]
For $a\ne0$, the constant $-1$ in \eqref{eq:Y-def} has zero Fourier coefficient.  Using \eqref{eq:qary-positive-product},
\begin{align*}
 \wh Y(a)
 &=\E_x\Phi(x)
   \E_bR_b^+(x)\overline{\psi(a\cdot b)}\\
 &=\E_x\Phi(x)\overline{\psi(a\cdot\Lambda(x))}\\
 &=\E_x\Phi(x)\overline{\chi_{Aa}(x)}
 =\wh\Phi(Aa).
\end{align*}
\end{proof}

\subsection{A level inequality on $\Fq^d$}

For $s\ge0$, let $H_s$ be the span of the characters
$b\mapsto\psi(a\cdot b)$ with $|a|_0=s$, and let $P_s$ be the orthogonal projection onto $H_s$.

\begin{lemma}\label{lem:qary-hypercontractive}
For every $r\ge2$, every $s\ge0$, and every $g\in H_s$,
\begin{equation}\label{eq:qary-hypercontractive}
 \norm g_r
 \le
 \left(
   C\sqrt{\frac{qr}{\log(2q)}}
 \right)^s
 \norm g_2,
\end{equation}
where $C$ is an absolute constant.
\end{lemma}

\begin{proof}
On the uniform one-coordinate probability space $\Fq$, consider the noise operator
\[
 T_\eta f:=\E f+\eta(f-\E f).
\]
Wolff's comparison theorem for finite probability spaces reduces the hypercontractive constant, up to an absolute factor, to the two-point space whose smaller atom has mass at least $1/q$; see \cite[Theorem~3.1]{Wolff}.  The two-point formula and its asymptotic forms in \cite[p.~221]{Wolff} then imply that for every $r\ge2$ there is an
\[
 \eta\ge c\sqrt{\frac{\log(2q)}{qr}}
\]
for which $T_\eta$ is a contraction from $L^2$ to $L^r$.  To see the two regimes explicitly, when $r\le\log q$ the relevant two-point constant is bounded below by a constant multiple of
$q^{-1/2+1/r}$, and
\[
 q^{-1/2+1/r}
 \ge c\sqrt{\frac{\log(2q)}{qr}}.
\]
When $r>\log q$, the two-point estimate is directly comparable to the displayed square root.  Enlarging the absolute constant handles the finitely many small values of $q$.

Tensorisation gives the same $L^2\to L^r$ contraction for the product operator on $\Fq^d$.  This operator acts by multiplication by $\eta^s$ on $H_s$.  Hence
\[
 \eta^s\norm g_r\le\norm g_2,
\]
which proves \eqref{eq:qary-hypercontractive}.  Applying the real-valued estimate to the real and imaginary parts changes only the absolute constant.
\end{proof}

\begin{lemma}\label{lem:qary-level}
Let $F:\Fq^d\to[0,1]$ and let $\E F=\mu$.  Then, for every integer $s\ge1$,
\begin{equation}\label{eq:qary-level}
 \norm{P_sF}_2^2
 \le
 \mu^2
 \left(
   C\frac{q}{\log(2q)}
   \log\frac e\mu
 \right)^s.
\end{equation}
The right-hand side is interpreted as zero when $\mu=0$.
\end{lemma}

\begin{proof}
The assertion is immediate if $\mu=0$.  Suppose that $\mu>0$.  By duality,
\[
 \norm{P_sF}_2
 =\sup_{g\in H_s,\ \norm g_2=1}
 \abs{\angles{F,g}}.
\]
Choose
\[
 r:=\max\set{2,\left\lceil\log\frac e\mu\right\rceil},
 \qquad
 p:=\frac r{r-1}.
\]
Since $0\le F\le1$,
\[
 \norm F_p\le\mu^{1/p}
 =\mu^{1-1/r}\le e\mu.
\]
H\"older's inequality and \cref{lem:qary-hypercontractive} give
\[
 \abs{\angles{F,g}}
 \le
 e\mu
 \left(
   C\sqrt{\frac{qr}{\log(2q)}}
 \right)^s.
\]
Squaring and using
$r\ll\log(e/\mu)$ proves \eqref{eq:qary-level}.
\end{proof}

Put
\begin{equation}\label{eq:L-def}
 L:=\log\left(\frac{eN}{\eps}\right)
 =\log\left(\frac{eNq^k}{S}\right).
\end{equation}

\begin{proposition}\label{prop:all-level-energy}
For every $s\ge1$,
\begin{equation}\label{eq:all-level-energy}
 \sum_{\substack{a\in\Fq^d\\|a|_0=s}}
 \abs{\wh\Phi(Aa)}^2
 \le
 \eps^2
 \left(
   C\frac{qL}{\log(2q)}
 \right)^s.
\end{equation}
\end{proposition}

\begin{proof}
Apply \cref{lem:qary-level} to $F=Y/N$, whose mean is
$\mu=\eps/N$.  Parseval on the $s$th phase level gives
\[
 \frac1{N^2}
 \sum_{|a|_0=s}\abs{\wh Y(a)}^2
 \le
 \frac{\eps^2}{N^2}
 \left(
   C\frac{qL}{\log(2q)}
 \right)^s.
\]
Cancel $N^{-2}$ and use
\eqref{eq:qary-phase-identification}.
\end{proof}

The factor $\eps^2$ in \eqref{eq:all-level-energy} is the essential gain: it is forced by the nonnegativity and small mean of the phase function $Y$.

\subsection{The negative product}

For $0\le\rho\le1$ and $y\in\Fq$, define
\begin{equation}\label{eq:qary-one-coordinate-negative}
 r_\rho(y)
 :=1-\frac{\rho}{q-1}
 \sum_{c\in\Fq^\times}\psi(cy).
\end{equation}
The elementary character sum gives
\begin{equation}\label{eq:qary-one-coordinate-values}
 r_\rho(y)=
 \begin{cases}
  1-\rho,&y=0,\\[1mm]
  1+\dfrac{\rho}{q-1},&y\ne0.
 \end{cases}
\end{equation}
In particular, $r_\rho\ge0$ and $\E_y r_\rho(y)=1$.  Set
\begin{equation}\label{eq:qary-negative-product}
 R_\rho^-(x)
 :=\prod_{i=1}^d
 r_\rho(\angles{m_i,x}).
\end{equation}
Linear independence gives
$R_\rho^-\ge0$ and $\E_xR_\rho^-=1$.

For $s\ge1$, write
\begin{equation}\label{eq:Qs-def}
 Q_s
 :=\sum_{\substack{a\in\Fq^d\\|a|_0=s}}
 \wh\Phi(Aa).
\end{equation}
All summands are nonnegative by \eqref{eq:quantisation}.  Every nonzero coefficient
$u=\wh\Phi(Aa)$ is at least $q^{-k}$, and hence
$u\le q^ku^2$.  Therefore \cref{prop:all-level-energy} gives
\begin{equation}\label{eq:Q-s}
 Q_s
 \le
 q^k\eps^2
 \left(
   C\frac{qL}{\log(2q)}
 \right)^s.
\end{equation}

Expanding \eqref{eq:qary-negative-product}, using
$\wh\Phi(-\xi)=\wh\Phi(\xi)$, and testing the pointwise inequality
$\Phi\ge1$ against this probability density gives
\begin{equation}\label{eq:qary-riesz-test}
 0\le
 \eps
 -\frac{\rho}{q-1}
 \sum_{i=1}^d\sum_{c\in\Fq^\times}
 \wh\Phi(cm_i)
 +\sum_{\substack{s\ge2\\s\text{ even}}}
 \left(\frac{\rho}{q-1}\right)^sQ_s.
\end{equation}
All odd levels greater than one are nonpositive and have been discarded.  Since
$cm_i\in U_{m_i}$ for every $c\ne0$,
\[
 \wh\Phi(cm_i)\ge q^{-k}.
\]
The first level in \eqref{eq:qary-riesz-test} is therefore at most
$-\rho dq^{-k}$.  If
\begin{equation}\label{eq:qary-rho-condition}
 C\rho\frac{L}{\log(2q)}\le\frac12,
\end{equation}
then \eqref{eq:Q-s} turns the even-level series into a geometric series.  Since $q/(q-1)\le2$, we obtain
\begin{equation}\label{eq:master-eps}
 0\le
 \eps-c\rho dq^{-k}
 +Cq^k\eps^2\rho^2
 \left(\frac{L}{\log(2q)}\right)^2.
\end{equation}
Multiplication by $q^k$ yields
\begin{equation}\label{eq:master-S}
 0\le
 S-c\rho d
 +CS^2\rho^2
 \left(\frac{L}{\log(2q)}\right)^2.
\end{equation}

\begin{proof}[Proof of \cref{thm:rank-intro}]
The positivity of $S$ was proved after
\eqref{eq:Phi-cover-mean}.  Put
\[
 M:=\frac{L}{\log(2q)}.
\]
If $d\le C_0M$, then
$d^{2/3}M^{-2/3}=O(1)$, and the conclusion follows from the integer bound $S\ge1$ after reducing the implied constant.

Assume that $d>C_0M$ and choose
\begin{equation}\label{eq:rho-choice-rank}
 \rho:=\eta d^{-1/3}M^{-2/3},
\end{equation}
where $\eta>0$ is a sufficiently small absolute constant.  Then
$C\rho M\le1/2$, so \eqref{eq:master-S} applies.  If, to the contrary,
\[
 S<c_0d^{2/3}M^{-2/3},
\]
then the first positive term in \eqref{eq:master-S} is at most
$c_0d^{2/3}M^{-2/3}$, while the quadratic term is at most
\[
 Cc_0^2\eta^2d^{2/3}M^{-2/3}.
\]
The negative term has size
$c\eta d^{2/3}M^{-2/3}$.  After fixing $\eta$ and then taking $c_0$ sufficiently small, the two positive terms are strictly smaller than the negative term, a contradiction.  This proves \eqref{eq:rank-intro}.
\end{proof}

\begin{remark}
The use of all $q$ phase values is essential for the improvement.  A binary or Steinhaus phase coordinate produces a level constant of order $L$, whereas the $q$-ary level estimate produces $qL/\log(2q)$.  The negative product contributes the compensating factor $(q-1)^{-s}$ at level $s$, leaving only $L/\log(2q)$.  For fixed $k$, this ratio is bounded, and the logarithmic loss disappears.
\end{remark}

\section{\texorpdfstring{$T$}{T}-compatible compression}\label{sec:compression}

The estimate in \cref{thm:rank-intro} is strongest when the linear rank is large.  We now show that a low-rank covering family can be compressed into the degree range where the theorem of Chow and Rimani\'c rules it out.  The compression must preserve multiplication by the powers
$1,T,\ldots,T^{k-1}$; an arbitrary linear isomorphism of coefficient spaces would not preserve the blocks $fA_k$.

\subsection{Rational approximation}

\begin{lemma}\label{lem:rational-approx}
Let $P\in\Fq[T]$ be monic and irreducible of degree $M$.  For every
$\alpha\in\Torus$, there is a polynomial $a$ of degree less than $M$ such that
\begin{equation}\label{eq:rational-approx}
 \norm{\alpha-a/P}\le q^{-(M+1)}.
\end{equation}
\end{lemma}

\begin{proof}
The first $M$ negative coefficients of $a/P$ depend linearly on the $M$ coefficients of $a$.  This linear map is injective.  Indeed, if the first $M$ negative coefficients vanished for a nonzero $a$, then
$\abs{a/P}\le q^{-(M+1)}$, whereas directly
\[
 \abs{a/P}=q^{\deg a-M}\ge q^{-M}.
\]
The map is therefore bijective.  We may choose $a$ so that the first $M$ negative coefficients of $a/P$ agree with those of $\alpha$, which is exactly \eqref{eq:rational-approx}.
\end{proof}

\subsection{Degree reduction modulo an irreducible polynomial}

\begin{proposition}\label{prop:compression}
Fix $k\ge2$.  Let $\cF$ be a family of distinct monic polynomials, set
\[
 d:=d(\cF),
 \qquad
 D:=\max_{f\in\cF}\deg f,
\]
and suppose that $\delta(\cF)<q^{-k}$.  Then there is a family $\cF'$ of the same cardinality and the same linear rank such that
\begin{equation}\label{eq:compression-conclusion}
 \delta(\cF')<q^{-k},
 \qquad
 \max_{f'\in\cF'}\deg f'<(k-1)d.
\end{equation}
\end{proposition}

\begin{proof}
We first give one degree-reduction step.  Suppose that
\begin{equation}\label{eq:h-compression}
 h:=\left\lfloor\frac Dd\right\rfloor\ge k-1.
\end{equation}
Put $M=D+1$, choose a monic irreducible polynomial $P$ of degree $M$, and work in the field
\[
 E:=\Fq[T]/(P)\cong\mathbb F_{q^M}.
\]
Let $b_1,\ldots,b_d$ be a basis of
$W:=\spanq\cF$.  Set
\[
 L:=M-h,
 \qquad
 V_L:=\set{g\bmod P:\deg g<L}\le E.
\]
The space $V_L$ has codimension $h$ in $E$.  For each $i$, multiplication by the nonzero residue class $b_i$ is an automorphism of $E$, so
$b_i^{-1}V_L$ also has codimension $h$.  Hence
\begin{equation}\label{eq:compression-intersection}
 \dim_{\Fq}\bigcap_{i=1}^db_i^{-1}V_L
 \ge M-dh\ge1.
\end{equation}
Choose a nonzero residue class $\lambda$ in this intersection.  Since every $f\in W$ is a linear combination of the $b_i$, the residue
$\lambda f\bmod P$ belongs to $V_L$.  Let $g_f$ be its unique representative of degree less than $L$.  Multiplication by $\lambda$ is injective on $W$, so the polynomials $g_f$ are distinct and their span has dimension $d$.  Moreover, two of the $g_f$ cannot be nontrivial scalar multiples: otherwise $g_f=cg_{f'}$ would imply $f=cf'$ because all degrees are less than $M$, and the monicity of $f,f'$ would force $c=1$ and $f=f'$.  Normalising each nonzero $g_f$ by a scalar in $\Fq^\times$ therefore gives a distinct monic family $\cF_1$ of the same size and rank, with
\begin{equation}\label{eq:one-step-degree}
 \max_{g\in\cF_1}\deg g\le D-h.
\end{equation}

It remains to show that the covering property is preserved.  Suppose, to the contrary, that there is an
$\alpha'\in\Torus$ satisfying
\begin{equation}\label{eq:good-time-compressed}
 \norm{\alpha'g}\ge q^{-k}
 \qquad(g\in\cF_1).
\end{equation}
By \cref{lem:rational-approx}, choose $a$ of degree less than $M$ with
\[
 \norm{\alpha'-a/P}\le q^{-(M+1)}.
\]
For $g\in\cF_1$, \eqref{eq:one-step-degree} gives
$\deg g\le M-h-1$, and therefore
\[
 \norm{(\alpha'-a/P)g}
 \le q^{-(h+2)}<q^{-k}.
\]
The ultrametric inequality and \eqref{eq:good-time-compressed} imply
\begin{equation}\label{eq:rational-time-good}
 \norm{(a/P)g}=\norm{\alpha'g}\ge q^{-k}.
\end{equation}

Choose a polynomial representative of $\lambda$.  For each original speed $f$, the corresponding normalised polynomial satisfies
\[
 g_f\equiv c_f\lambda f\pmod P
 \qquad(c_f\in\Fq^\times).
\]
It follows that
\[
 \frac aP g_f
 \equiv
 c_f\frac{a\lambda}{P}f
 \pmod{\Fq[T]}.
\]
Let $\alpha$ be the fractional part of $a\lambda/P$.  Since scalar multiplication by $c_f$ does not change the norm, \eqref{eq:rational-time-good} gives
\[
 \norm{\alpha f}\ge q^{-k}
 \qquad(f\in\cF),
\]
contradicting $\delta(\cF)<q^{-k}$.  Thus
$\delta(\cF_1)<q^{-k}$.

Repeat the reduction while \eqref{eq:h-compression} holds.  At termination,
$\lfloor D/d\rfloor<k-1$, and hence $D<(k-1)d$.  The size and rank are preserved at every step, proving the proposition.
\end{proof}

\begin{remark}
The argument resembles a Freiman-model compression, but the reduction is carried out inside the field
$\Fq[T]/(P)$.  It therefore respects multiplication by $T$ modulo $P$, which is essential because the covering kernel attached to $f$ depends on the whole block
$f,Tf,\ldots,T^{k-1}f$.
\end{remark}

\subsection{The small-degree range}

Let $I_m(q)$ denote the number of monic irreducible polynomials of degree $m$ over $\Fq$.  We use the following theorem of Chow and Rimani\'c in the form stated in
\cite[Theorem~1.6]{ChowRimanic}.

\begin{theorem}[Chow--Rimani\'c]\label{thm:CR-small-degree}
Let $k>1$, and let $\cF$ be a family of nonzero polynomials of degree at most $D$.  If
\begin{equation}\label{eq:CR-small-degree-condition}
 \frac{I_{k+1}(q)}{q^k+q^{k-1}+\cdots+q}
 >\left\lfloor\frac D{k+1}\right\rfloor,
\end{equation}
then every family with $\abs{\cF}<Q_k(q)$ satisfies
$\delta(\cF)\ge q^{-k}$.
\end{theorem}

The exact condition immediately gives a convenient uniform corollary.

\begin{lemma}\label{lem:small-degree-cor}
Fix $k\ge2$.  For all sufficiently large $q$, if
\begin{equation}\label{eq:q-fourth-degree}
 \max_{f\in\cF}\deg f\le q/4
 \qquad\text{and}\qquad
 \abs{\cF}<Q_k(q),
\end{equation}
then $\delta(\cF)\ge q^{-k}$.
\end{lemma}

\begin{proof}
The standard formula for irreducible polynomials gives
\[
 I_{k+1}(q)
 =\frac{q^{k+1}}{k+1}
 +O_k(q^{(k+1)/2}).
\]
Consequently,
\[
 \frac{I_{k+1}(q)}{q^k+\cdots+q}
 =\frac q{k+1}+O_k(1),
\]
and in particular the ratio is greater than
$q/[2(k+1)]$ for sufficiently large $q$.  Under
\eqref{eq:q-fourth-degree},
\[
 \left\lfloor\frac D{k+1}\right\rfloor
 \le\frac q{4(k+1)}.
\]
Thus \eqref{eq:CR-small-degree-condition} holds, and
\cref{thm:CR-small-degree} applies.
\end{proof}

\begin{corollary}\label{cor:rank-lower}
Fix $k\ge2$.  For all sufficiently large $q$, if
\[
 \abs{\cF}<Q_k(q)
 \qquad\text{and}\qquad
 \delta(\cF)<q^{-k},
\]
then
\begin{equation}\label{eq:rank-lower}
 d(\cF)\ge\frac{q}{4(k-1)}.
\end{equation}
\end{corollary}

\begin{proof}
Apply \cref{prop:compression}.  If
$d<q/[4(k-1)]$, the compressed family has maximum degree less than
$(k-1)d<q/4$, contradicting \cref{lem:small-degree-cor}.
\end{proof}

\begin{proof}[Proof of \cref{thm:fixed-k-intro}]
Let $\cF$ be a covering family.  Replacing every speed by its monic representative and deleting repetitions preserves the cover and can only decrease its cardinality, so it suffices to treat a family of distinct monic polynomials.  Put
$N=\abs{\cF}$ and $S=N-q^k$.  If $N\ge Q_k(q)$, then
$S\ge q^{k-1}$, which is stronger than the required estimate for sufficiently large $q$.

Suppose that $N<Q_k(q)$.  By \cref{cor:rank-lower},
$d(\cF)\gg_kq$.  Moreover,
$N<Q_k(q)\ll_kq^k$ and $S\ge1$, so
\[
 L=\log\left(\frac{eNq^k}{S}\right)
 \ll_k\log q.
\]
Consequently,
\[
 \frac{\log(2q)}{L}\gg_k1.
\]
Applying \cref{thm:rank-intro} gives
\[
 S\gg_k d(\cF)^{2/3}\gg_kq^{2/3},
\]
which proves \eqref{eq:fixed-k-intro}.
\end{proof}

\section{An integer-multiplicity refinement for \texorpdfstring{$k=2$}{k=2}}\label{sec:k2}

Throughout this section $k=2$.  If the maximum speed degree is $D$, then each kernel $K_f$ has dimension $D$ in the ambient space $V_{D+2}$.  We retain the ordering and sunflower decomposition of Chow and Rimani\'c, but sharpen the estimate for the union of pairwise intersections inside a new kernel.

\subsection{An integer-valued support inequality}

The function $\vartheta$ from \eqref{eq:vartheta-def} is continuous and decreasing.  Formula \eqref{eq:vartheta-piecewise} follows by comparing two consecutive affine functions in the maximum: the functions indexed by $h$ and $h+1$ cross at $u=h$.  In particular, $\vartheta$ is the piecewise-linear interpolation of $u\mapsto1/(1+u)$ at the nonnegative integers.  Since $u\mapsto1/(1+u)$ is convex,
\begin{equation}\label{eq:vartheta-CS-comparison}
 \vartheta(u)\ge\frac1{1+u},
\end{equation}
with strict inequality unless $u$ is an integer.

\begin{lemma}\label{lem:integer-support}
Let $X$ be a finite set and let $Z:X\to\mathbb Z_{\ge0}$.  Put
\[
 M_1:=\sum_{x\in X}Z(x),
 \qquad
 M_2:=\sum_{x\in X}Z(x)(Z(x)-1).
\]
If $M_1>0$, then
\begin{equation}\label{eq:integer-support}
 \abs{\set{x\in X:Z(x)>0}}
 \ge M_1\vartheta\left(\frac{M_2}{M_1}\right).
\end{equation}
\end{lemma}

\begin{proof}
For every $h\in\mathbb N$ and every integer $z\ge0$,
\begin{equation}\label{eq:integer-pointwise}
 \one_{\{z>0\}}
 \ge
 \frac{2z}{h+1}
 -\frac{z(z-1)}{h(h+1)}.
\end{equation}
For $z=0$ both sides vanish.  For $z\ge1$, the difference between the two sides is
\[
 \frac{(z-h)(z-h-1)}{h(h+1)}\ge0,
\]
because $z$ is an integer.  Summing \eqref{eq:integer-pointwise} over $X$ gives
\[
 \abs{\set{Z>0}}
 \ge
 \frac{2M_1}{h+1}
 -\frac{M_2}{h(h+1)}.
\]
Maximising over $h$ gives \eqref{eq:integer-support}.
\end{proof}

The next two consequences are the forms used in the covering argument.

\begin{lemma}\label{lem:integer-basic}
Let $V_1,\ldots,V_t$ be distinct subspaces of a $D$-dimensional vector space over $\Fq$.  Suppose that
\[
 \abs{V_i}=q^{D-2},
 \qquad
 \abs{V_i\cap V_j}\le q^{D-3}
 \quad(i\ne j).
\]
Then
\begin{equation}\label{eq:integer-basic}
 \abs{\bigcup_{i=1}^tV_i}
 \ge
 tq^{D-2}\vartheta\left(\frac{t-1}{q}\right).
\end{equation}
\end{lemma}

\begin{proof}
Set $Z:=\sum_{i=1}^t\one_{V_i}$.  Then
\[
 M_1=tq^{D-2},
 \qquad
 M_2
 =\sum_{i\ne j}\abs{V_i\cap V_j}
 \le t(t-1)q^{D-3}.
\]
For each fixed $h$, the right-hand side obtained by summing
\eqref{eq:integer-pointwise} is increasing in $M_1$ and decreasing in $M_2$.  Substituting the displayed bounds and maximising over $h$ proves \eqref{eq:integer-basic}.
\end{proof}

Recall that a family of subspaces is a sunflower with core $L$ if every two distinct members meet in the same subspace $L$.

\begin{lemma}\label{lem:integer-sunflower}
Under the hypotheses of \cref{lem:integer-basic}, suppose in addition that
$V_1,\ldots,V_r$ form a sunflower with common core $L$ and
$\abs L\le q^{D-3}$.  If $r\le q$ and $t\ge r$, then
\begin{align}
 \abs{\bigcup_{i=1}^tV_i}
 \ge{}&rq^{D-2}-(r-1)q^{D-3}\notag\\
 &+(t-r)(q-r)q^{D-3}
 \vartheta\left(\frac{t-r-1}{q-r}\right).
 \label{eq:integer-sunflower}
\end{align}
The last term is interpreted as zero when $t=r$ or $r=q$.
\end{lemma}

\begin{proof}
Let $A:=\bigcup_{i=1}^rV_i$.  Since the first $r$ subspaces form a sunflower,
\[
 \abs A
 =rq^{D-2}-(r-1)\abs L
 \ge rq^{D-2}-(r-1)q^{D-3}.
\]
Write $s=t-r$ and
$B_j:=V_{r+j}\setminus A$ for $1\le j\le s$.  The union bound and the pairwise-intersection hypothesis give
\[
 \abs{B_j}\ge(q-r)q^{D-3},
 \qquad
 \abs{B_i\cap B_j}\le q^{D-3}\quad(i\ne j).
\]
For $W:=\sum_{j=1}^s\one_{B_j}$, we therefore have
\[
 \sum_xW(x)\ge s(q-r)q^{D-3},
 \qquad
 \sum_xW(x)(W(x)-1)\le s(s-1)q^{D-3}.
\]
Apply the pointwise inequality \eqref{eq:integer-pointwise} with an arbitrary $h$, substitute these two estimates, and then maximise over $h$.  This gives
\[
 \abs{\bigcup_{j=1}^sB_j}
 \ge
 s(q-r)q^{D-3}
 \vartheta\left(\frac{s-1}{q-r}\right).
\]
The sets $B_j$ are disjoint from $A$, so adding the two contributions proves the lemma.
\end{proof}

\subsection{The Chow--Rimani\'c ordering}

We record the precise consequences of the ordering argument in
\cite[Section~4]{ChowRimanic} that are needed below.  The \emph{covering contribution} of the $m$th kernel in an ordering is
\[
 \abs{K_m\setminus\bigcup_{i<m}K_i}.
\]

\begin{lemma}\label{lem:CR-ordering}
Suppose that the kernels associated with a family $\cF$ cover $V_{D+2}$.  Let $r\le q$ be the maximum size of a codimension-four sunflower and order the kernels as in the proof of
\cite[Proposition~4.8]{ChowRimanic}.  Then the following statements hold.
\begin{enumerate}[label=\textup{(\roman*)}]
\item The total covering contribution of the first
$1+r(r-1)$ kernels is at most
\begin{align}
 S_{\mathrm{init}}:={}&q^D+(r-1)(q^D-q^{D-2})\notag\\
 &+(r-1)^2
 \bigl(q^D-rq^{D-2}+(r-1)q^{D-3}\bigr).
 \label{eq:S-init}
\end{align}
\item Before the change point, if $K_m$ is the new kernel and
$V_i:=K_i\cap K_m$, then at least $t$ of the $V_i$ are distinct whenever
\begin{equation}\label{eq:CR-pigeonhole}
 m-1>(t-1)(r-1).
\end{equation}
Each such intersection has size $q^{D-2}$, and two distinct intersections meet in at most $q^{D-3}$ points.
\item The first $r$ distinct intersections in \textup{(ii)} form a sunflower inside $K_m$, with core of size at most $q^{D-3}$.
\item After the change point, every new kernel meets the union of its predecessors in at least $q^{D-1}$ points.
\end{enumerate}
\end{lemma}

\begin{proof}
Statement \textup{(i)} is equation~(4.9) in the proof of
\cite[Proposition~4.8]{ChowRimanic}.  Statements \textup{(ii)} and \textup{(iii)} are the two conclusions of Claim~1 in that proof, and \textup{(iv)} is the definition of the change point.  The same proof also shows that any lower bound from the pre-change analysis remains valid after the change point whenever it is at most $q^{D-1}$.
\end{proof}

\subsection{The asymptotic covering loss}

Consider a sequence of covering families for which
\begin{equation}\label{eq:R-r-asymptotic}
 R:=\abs{\cF}=q^2+cq+o(q),
 \qquad
 r=\lambda q+o(q),
 \qquad
 0\le\lambda\le1.
\end{equation}
The nominal capacity of one kernel is $q^D$.  Comparing
\eqref{eq:S-init} with the nominal contribution
$[1+r(r-1)]q^D$, the initial phase loses
\begin{equation}\label{eq:initial-saving}
 \lambda^3q^{D+1}+o(q^{D+1}).
\end{equation}

For a later kernel, suppose that the ordering guarantees
$t=xq+o(q)$ distinct intersections with its predecessors.  By
\cref{lem:integer-basic}, their union inside the new kernel has size at least
\begin{equation}\label{eq:atilde-asymptotic}
 q^{D-1}\bigl(\widetilde a(x)+o(1)\bigr).
\end{equation}
By \cref{lem:integer-sunflower}, it is also at least
\begin{equation}\label{eq:btilde-asymptotic}
 q^{D-1}\bigl(\widetilde b_\lambda(x)+o(1)\bigr).
\end{equation}
Indeed,
\[
 \frac{t-r-1}{q-r}
 =\frac{x-\lambda}{1-\lambda}+o(1).
\]
For $u\in[h-1,h]$, formula \eqref{eq:vartheta-piecewise} gives
$u\vartheta(u)\le h/(h+1)<1$.  Consequently,
\[
 0\le\widetilde a(x)\le1,
 \qquad
 0\le\widetilde b_\lambda(x)\le
 \lambda+(1-\lambda)^2\le1.
\]
Thus the same lower bounds remain valid after the change point by
\cref{lem:CR-ordering}(iv).

Condition \eqref{eq:CR-pigeonhole} shows that increasing the guaranteed value of $t$ by one uses at most $r-1=\lambda q+o(q)$ further kernels.  Starting at the end of the initial phase and continuing to
$R=q^2+O(q)$, the scaled variable $x=t/q$ runs from
$\lambda$ to $1/\lambda$, up to $o(1)$ endpoint errors.

\begin{proposition}\label{prop:Ctilde-lambda}
Suppose that $0<\lambda<1$ in \eqref{eq:R-r-asymptotic}.  Then
\begin{equation}\label{eq:c-lower-Ctilde-lambda}
 c\ge\widetilde{\cC}(\lambda).
\end{equation}
The same conclusion holds at $\lambda=0$ and $\lambda=1$ after the continuous extension
$\widetilde{\cC}(0)=\widetilde{\cC}(1)=1$.
\end{proposition}

\begin{proof}
The total nominal capacity of the $R$ kernels is $Rq^D$.  The initial loss is given by \eqref{eq:initial-saving}.  In each block for which the guaranteed number of distinct intersections is
$t=xq+o(q)$, the overlap loss per kernel is at least
\[
 q^{D-1}
 \max\set{\widetilde a(x),\widetilde b_\lambda(x)}
 +o(q^{D-1}),
\]
and the block contains $\lambda q+o(q)$ kernels.  Summing the blocks gives a further loss
\[
 \lambda q^{D+1}
 \int_\lambda^{1/\lambda}
 \max\set{\widetilde a(x),\widetilde b_\lambda(x)}\,dx
 +o(q^{D+1}).
\]
A cover requires total contribution at least $q^{D+2}$, whereas
\[
 Rq^D=q^{D+2}+cq^{D+1}+o(q^{D+1}).
\]
Comparing the coefficient of $q^{D+1}$ gives
\eqref{eq:c-lower-Ctilde-lambda}.

At $\lambda=0$ and $\lambda=1$, the conclusion follows from the corresponding finite sums, or by taking the continuous limits of the displayed expression.
\end{proof}

\begin{proof}[Proof of \cref{thm:c2-intro}]
Choose a sequence of covering families that approaches the lower limit in
\eqref{eq:c2-intro}.  If the normalised excess is unbounded, there is nothing to prove, so we may assume
$R=q^2+O(q)$.  Pass to a subsequence along which the maximum sunflower size divided by $q$ converges.

If the maximum codimension-four sunflower has size at most $q$, then
\cref{prop:Ctilde-lambda} gives
$c\ge\min_{0\le\lambda\le1}\widetilde{\cC}(\lambda)=\widetilde c_2$.  If the maximum sunflower has size exactly $q+1$, the unsimplified estimate in the proof of
\cite[Proposition~4.6]{ChowRimanic} gives
$R>q^2+q$.  If it has size at least $q+2$, then
\cite[Proposition~A.1]{ChowRimanic} gives
$R>q^2+q+1$ for $q>8$.  Thus the range covered by
\cref{prop:Ctilde-lambda} is the only possible asymptotic bottleneck, and the variational lower bound follows.
\end{proof}

\subsection{The variational constant}

The functions $\widetilde a$ and $\widetilde b_\lambda$ are piecewise quadratic, with breakpoints at integers in their respective arguments.  Hence
$\widetilde{\cC}(\lambda)$ is an explicit piecewise elementary function.

\begin{lemma}\label{lem:c2tilde-rigorous-lower}
One has
\begin{equation}\label{eq:c2tilde-rigorous-lower}
 \widetilde c_2>0.5278.
\end{equation}
\end{lemma}

\begin{proof}
By \eqref{eq:vartheta-CS-comparison},
$\widetilde a(x)\ge x/(1+x)$.  Therefore
\begin{align*}
 \widetilde{\cC}(\lambda)
 &\ge F(\lambda)
 :=\lambda^3+
 \lambda\int_\lambda^{1/\lambda}\frac{x}{1+x}\,dx\\
 &=1-\lambda^2+\lambda^3+\lambda\log\lambda.
\end{align*}
Moreover,
\[
 F''(\lambda)=-2+6\lambda+\lambda^{-1}
 \ge2\sqrt6-2>2.89.
\]
Put $\lambda_0=12/25$.  The elementary bounds
$-0.734<\log(12/25)<-0.7339$ give
\[
 F(\lambda_0)>0.527872,
 \qquad
 \abs{F'(\lambda_0)}<0.0028.
\]
Strong convexity therefore implies
\[
 \inf_{0<\lambda\le1}F(\lambda)
 \ge
 F(\lambda_0)-
 \frac{F'(\lambda_0)^2}{2(2.89)}
 >0.5278.
\]
Taking the infimum proves the claim.
\end{proof}

The inequality \eqref{eq:vartheta-CS-comparison} also shows directly why this improves the ordinary Cauchy--Schwarz argument.  Replacing $\vartheta(u)$ by $1/(1+u)$ recovers the two former envelopes
\[
 \frac{x}{1+x}
 \quad\text{and}\quad
 \lambda+
 \frac{(x-\lambda)(1-\lambda)^2}{1+x-2\lambda}.
\]
The inequality is strict for almost every $x$ in every nondegenerate integration interval.  Thus the new variational constant is strictly larger than the constant obtained from Cauchy--Schwarz alone.

A numerical minimisation gives
\[
 \widetilde\lambda_*=0.4067490\ldots,
 \qquad
 \widetilde c_2=0.5829944375\ldots.
\]
The accompanying script \texttt{evaluate\_c2\_constant.py} evaluates the integral after splitting at all piecewise-linear breakpoints and then applies a one-dimensional golden-section search.  No numerical approximation is used in the proof of the exact variational statement.

\section{Low-degree syzygies and relation packets}\label{sec:packets}

The fixed-$k$ bound in \cref{thm:fixed-k-intro} remains below the natural second term $q^{k-1}$.  We now identify the first obstruction encountered by a second-moment covering argument at that scale.

\subsection{Relation spaces and triple intersections}

Recall that $A_k=\Fq[T]_{<k}$.  For three nonzero polynomials $m,f,g$, define
\begin{equation}\label{eq:Rk-def}
 \cR_k(m,f,g)
 :=\set{(A,B,C)\in A_k^3:Am+Bf+Cg=0}.
\end{equation}
The image of the linear map
\[
 A_k^3\longrightarrow\Fq[T],
 \qquad
 (A,B,C)\longmapsto Am+Bf+Cg,
\]
is $U_m+U_f+U_g$.  Therefore
\begin{equation}\label{eq:rank-relation}
 \dim_{\Fq}(U_m+U_f+U_g)
 =3k-\dim_{\Fq}\cR_k(m,f,g).
\end{equation}
If all three speeds have degree at most $D$, orthogonality in
$V_{D+k}$ gives
\begin{equation}\label{eq:triple-intersection}
 \abs{K_m\cap K_f\cap K_g}
 =q^{D-2k+\dim\cR_k(m,f,g)}.
\end{equation}
Thus two independent relations enlarge the generic triple intersection by at least a factor $q^2$.

\subsection{Classification of the first obstruction}

\begin{proof}[Proof of \cref{thm:syzygy-intro}]
Choose linearly independent vectors
\[
 u=(A_1,B_1,C_1),
 \qquad
 v=(A_2,B_2,C_2)
\]
in $\cR_k(m,f,g)$.

Suppose first that $u$ and $v$ are linearly dependent over
$\Fq(T)$.  Let
$w_0\in\Fq[T]^3$ be a primitive polynomial vector spanning their common rational line.  By Gauss's lemma, there are polynomials
$a,b\in\Fq[T]$ such that
$u=aw_0$ and $v=bw_0$.  Since $u$ and $v$ are not scalar multiples over $\Fq$, at least one of $a,b$ is nonconstant.  All coordinates of the corresponding multiple have degree at most $k-1$, so every nonzero coordinate of $w_0$ has degree at most $k-2$.  The identity
$w_0\cdot(m,f,g)=0$ gives alternative \textup{(i)}.

Suppose now that $u$ and $v$ are independent over $\Fq(T)$.  Their cross product
\[
 w:=u\times v
\]
is nonzero.  Each coordinate of $w$ is a $2\times2$ determinant of polynomials of degree at most $k-1$, and hence has degree at most $2k-2$.  Both $w$ and $(m,f,g)$ span the one-dimensional orthogonal complement of
$\operatorname{span}_{\Fq(T)}\{u,v\}$, so
\[
 (m,f,g)=h w
\]
for some $h\in\Fq(T)$.  Divide $w$ by the greatest common divisor of its coordinates and call the resulting primitive vector $w_0=(m_0,f_0,g_0)$.  Then
$(m,f,g)=h_0w_0$ for some $h_0\in\Fq(T)$.  Since $h_0w_0$ has polynomial coordinates and $w_0$ is primitive, Gauss's lemma implies that $h_0$ is a polynomial.  Setting $P=h_0$ gives
\[
 m=Pm_0,
 \qquad
 f=Pf_0,
 \qquad
 g=Pg_0,
\]
with $\deg m_0,\deg f_0,\deg g_0\le2k-2$.  This is alternative \textup{(ii)}.
\end{proof}

For example, when $k=3$, every triple with two independent syzygies has one of the forms
\[
 Am+Bf+Cg=0,
 \qquad
 \deg A,\deg B,\deg C\le1,
\]
or
\[
 m=Pm_0,
 \qquad
 f=Pf_0,
 \qquad
 g=Pg_0,
 \qquad
 \deg m_0,\deg f_0,\deg g_0\le4.
\]
The first is a low-degree labelled relation; the second is a common-factor packet with bounded quotients.

\subsection{A packet-free half-way theorem}

We formulate a quantitative condition under which the harmful triples from
\cref{thm:syzygy-intro} contribute negligibly to the second moment.

\begin{definition}\label{def:packet-free}
Fix $k$.  A sequence of covering families $\cF_q$, with maximum degrees $D_q$ and
\[
 \abs{\cF_q}=q^k+O(q^{k-1}),
\]
is called \emph{packet-free} if the associated kernels can be ordered as
$K_1,K_2,\ldots$ so that, for all but $o(q^k)$ indices $m$, there is a set
$I_m\subset\{1,\ldots,m-1\}$ satisfying the following conditions:
\begin{enumerate}[label=\textup{(\roman*)}]
\item $\abs{I_m}=(1-o(1))q^{k-1}$;
\item the intersections
\[
 V_{m,i}:=K_m\cap K_i,
 \qquad i\in I_m,
\]
are pairwise distinct and have the generic size
$\abs{V_{m,i}}=q^{D_q-k}$;
\item the number of unordered pairs
$\{i,j\}\subset I_m$ for which
\[
 \dim_{\Fq}\cR_k(f_m,f_i,f_j)\ge2
\]
is $o(q^k)$.
\end{enumerate}
All asymptotic notation is for $q\to\infty$ with $k$ fixed.
\end{definition}

The condition does not exclude every low-degree relation.  It only requires that, after selecting about $q^{k-1}$ distinct intersections in a typical new kernel, pairs with two independent syzygies have negligible total second-moment mass.

\begin{theorem}\label{thm:packet-free}
Every packet-free sequence of covering families satisfies
\begin{equation}\label{eq:packet-free-conclusion}
 \abs{\cF_q}
 \ge q^k+\left(\frac12-o(1)\right)q^{k-1}.
\end{equation}
\end{theorem}

\begin{proof}
Fix a typical index $m$ from \cref{def:packet-free} and abbreviate
$t:=\abs{I_m}$ and $V_i:=V_{m,i}$.  On $K_m$, define
\[
 Z(x):=\sum_{i\in I_m}\one_{V_i}(x).
\]
The first moment is
\begin{equation}\label{eq:packet-first}
 \sum_{x\in K_m}Z(x)
 =tq^{D_q-k}
 =(1-o(1))q^{D_q-1}.
\end{equation}

For a pair that is not harmful, \eqref{eq:triple-intersection} and
$\dim\cR_k\le1$ give
\[
 \abs{V_i\cap V_j}
 \le q^{D_q-2k+1}.
\]
For a harmful pair, the distinctness of $V_i$ and $V_j$ implies that their intersection has codimension at least $k+1$ inside $K_m$, and hence
\[
 \abs{V_i\cap V_j}
 \le q^{D_q-k-1}.
\]
There are only $o(q^k)$ harmful pairs.  It follows that
\begin{align*}
 \sum_{x\in K_m}Z(x)^2
 &\le tq^{D_q-k}
 +t^2q^{D_q-2k+1}
 +o(q^k)q^{D_q-k-1}\\
 &\le(2+o(1))q^{D_q-1}.
\end{align*}
Cauchy--Schwarz and \eqref{eq:packet-first} yield
\begin{equation}\label{eq:packet-overlap}
 \abs{\bigcup_{i\in I_m}V_i}
 \ge\left(\frac12-o(1)\right)q^{D_q-1}.
\end{equation}
Thus, for all but $o(q^k)$ kernels, the newly covered portion has size at most
\[
 q^{D_q}-
 \left(\frac12-o(1)\right)q^{D_q-1}.
\]
The exceptional kernels contribute at most $q^{D_q}$ points each.  Summing the sequential covering contributions and using
$\abs{V_{D_q+k}}=q^{D_q+k}$ gives
\[
 q^{D_q+k}
 \le
 \abs{\cF_q}q^{D_q}
 -\left(\frac12-o(1)\right)q^{D_q+k-1}.
\]
Division by $q^{D_q}$ proves
\eqref{eq:packet-free-conclusion}.
\end{proof}

\begin{remark}
The theorem is a statement about packet-free sequences, not an unconditional assertion about $C_k(q)$.  It implies the same lower bound for $C_k(q)$ if a sequence of asymptotically extremal covering families is packet-free.  Removing that hypothesis is precisely the agreement problem discussed next.
\end{remark}

\subsection{Why local relations are not yet global packets}

The covering multiplicity has the exact moment identity
\begin{equation}\label{eq:moment-rank}
 \E\Phi^r
 =\sum_{f_1,\ldots,f_r\in\cF}
 q^{-\dim(U_{f_1}+\cdots+U_{f_r})}.
\end{equation}
Indeed, the expectation of
$\one_{K_{f_1}}\cdots\one_{K_{f_r}}$ is the density of the annihilator of
$U_{f_1}+\cdots+U_{f_r}$.  Thus high moments of a near-union-bound cover force many tuples with unexpectedly small block rank.

The rank defect alone is not enough.  Each relation has coefficients in
$A_k$, and these coefficient labels may vary from tuple to tuple.  Many independently labelled local relations can reproduce the same moment numerology without producing a common factor, a common low-dimensional block, or a global recurrence.  An unconditional half-way theorem would require a labelled agreement result: a positive density of triples with
$\dim\cR_k(m,f,g)\ge2$ should decompose into a controlled collection of coherent packets of the two types in
\cref{thm:syzygy-intro}, with common cores charged only once in the global covering count.

This is the function-field form of the relation-label compatibility problem that appears in higher-moment approaches to the classical Lonely Runner Conjecture.  The advantage of the present model is that the labels lie in the fixed finite-dimensional space $A_k$ and the first obstruction admits the explicit classification above.

\section{Further questions}\label{sec:open}

\subsection{A corrected large-field conjecture}

The counterexample in \cref{sec:counterexample} rules out the original unrestricted formulation, but it remains natural to ask whether the projective threshold is correct once the coefficient field is sufficiently large.

\begin{conjecture}\label{conj:asymptotic-CR}
For every fixed $k\ge1$, there is a $q_0(k)$ such that
\[
 C_k(q)=Q_k(q)
\]
for every prime power $q\ge q_0(k)$.
\end{conjecture}

A weaker first target is the second-term estimate
\[
 C_k(q)\ge q^k+(1-o(1))q^{k-1}.
\]
For $k=2$, \cref{thm:c2-intro} gives a positive explicit proportion but not the full coefficient one.  For $k\ge3$, even a uniform positive constant multiplying $q^{k-1}$ would go beyond \cref{thm:fixed-k-intro}.

\subsection{Beyond the exponent $2/3$}

The $q$-ary phase space removes the logarithmic loss from the fixed-$k$ theorem, but the exponent $2/3$ remains built into the present master inequality
\[
 0\le S-c\rho d+CS^2\rho^2.
\]
Balancing the three terms gives $S^3\gtrsim d^2$.  A better power therefore requires an input not contained in rank and Fourier quantisation alone.  Possible routes include showing that the relevant Fourier multiplicities are usually larger than one, making part of the quadratic level favourable, or using coherent relations among dependent speed directions rather than only a linearly independent subfamily.

\subsection{The small exceptional cases}

The explicit construction proves only that $C_3(2)\le13$.

\begin{problem}
Determine $C_3(2)$ exactly and classify all failures of the projective threshold for small pairs $(q,k)$.
\end{problem}

The finite problem may be formulated as set cover on the collection of partial-circulant kernels.  Beyond determining isolated values, the geometry of extremal certificates may indicate which algebraic packets are responsible for failure at small characteristic.

\subsection{Stability at $k=2$}

The function $\vartheta$ gives the optimal consequence of the first two factorial moments of an integer-valued multiplicity.  A further substantial improvement of $\widetilde c_2$ must therefore use additional geometry or higher moments.  Near equality in \cref{lem:integer-support} constrains the multiplicity to concentrate on two adjacent integers, while near-maximal pair intersections force coherent codimension-five labels.  A useful stability theorem should separate an energy-dispersed regime, in which a third factorial moment improves the variational envelope, from a structured regime, in which the intersections form a polynomial packet that can be counted more sharply.

\subsection{Block phases}

The phase product in \cref{sec:phase} assigns one $\Fq$-valued phase coordinate to each selected speed.  A stronger construction would attach phases to the full block
\[
 U_f=fA_k.
\]
Its Fourier coefficients would record relations
\[
 A_1f_1+\cdots+A_rf_r=0,
 \qquad
 A_i\in A_k,
\]
directly.  The obstacle is that these same relations create additional constant terms and destroy automatic normalisation.  Possible approaches include quotienting by a selected relation module or averaging over a correlated phase law whose positive-definite Fourier transform vanishes on specified harmful relation classes.

\subsection{Growing $k$}

The $q$-ary level estimate remains meaningful when $k$ grows with $q$.  The fixed-$k$ restriction enters through the small-degree theorem and through constants in the packet analysis.  A uniform theory would require simultaneous control of the compression endpoint, the dependence of the covering rank on $k$, and the number and geometry of low-degree syzygy packets.

\section*{Code availability}

The files \texttt{verify\_counterexample.py} and
\texttt{evaluate\_c2\_constant.py} accompany this manuscript.  They use only the Python standard library.  The first exhaustively verifies the thirteen-kernel covering certificate; the second evaluates the variational constant in \eqref{eq:c2tilde}.

\section*{Acknowledgements}

The author thanks J\"org Wills for clarifying the early history of the Lonely Runner Conjecture and the origin of the term ``view obstruction.''  The author also thanks Benjamin Bedert for helpful discussions that led to a deeper appreciation of the power of Riesz-product methods in problems of this kind.

\section*{Statement on the use of AI}

OpenAI's ChatGPT was used during exploratory work and preparation of this manuscript to assist with calculations and their verification, code prototyping, literature organisation, exposition, and LaTeX editing.  The author checked the mathematical arguments and takes full responsibility for every claim and for the final contents of the paper.

\end{document}